\documentclass[12pt,reqno]{amsart}
\usepackage{amssymb}
\usepackage[english]{babel}
\usepackage{enumerate}
\makeatletter
\@namedef{subjclassname@2010}{%
 \textup{2010} Mathematics Subject Classification}
\makeatother
\newtheorem{theorem}{Theorem}[section]

\newtheorem{lemma}{Lemma}[section]
\newtheorem{conjecture}{Conjecture}

\newtheorem{remark}{Remark}[section]

\begin{document}

\title[Proof of the conjecture]{Proof of the conjecture of Keskin, Siar and Karaatli}

\author[R. Boumahdi]{Rachid Boumahdi}
\address{{Laboratoire d'Arithm\'{e}tique, Codage, Combinatoire et Calcul Formel}\\
{Universit\'{e} des Sciences et Technologies Houari Boumedi\`{e}ne}\\
{16024, El Alia, Alger, Alg\'{e}rie}\\}
\email{r\_boumehdi@esi.dz}

\author[O. Kihel]{Omar Kihel}
\address{Department of Mathematics, Brock University, Ontario, Canada L2S 3A1}
\email{okihel@brocku.ca}

\author[S. Mavecha]{Sukrawan Mavecha}
\address{Department of Mathematics, Faculty of Science, King Mongkut's Institute of Technology Ladkrabang, Bangkok, 10520, Thailand}
\email{ktsukraw@kmitl.ac.th}
\maketitle

%\subjclass[2010]{11D09}

%\keywords{Quadratic diophantine equations, Pell equation, conjecture of Keskin, Olcay and Siar}

\begin{abstract}
In this paper among other results, we will prove the conjecture of Keskin, Siar and Karaatli on the Diophantine equation $x^{2}-kxy+y^{2}-2^{n}=0$.
\end{abstract}

\section{Introduction}

There have been much recent interest in the Diophantine equation
\begin{equation}\label{eq1}
x^{2}-kxy+y^{2}+lx=0
\end{equation}
for different values of the integers $k$ and $l.$ Marlewski and Zarzycki \cite{M-Z}, considered equation (\ref{eq1}) for $l=1$,
and proved that equation (\ref{eq1}) has no positive solution for $l=1$ and $
k>3$, but has an infinite number of solutions for $k=3$ and $l=1$. Keskin et al. in  \cite{K-K-SS} and \cite{K-K-S} considered equation (\ref{eq1}) for $l=-1$ and proved that it has positive integer solutions for $k>1.$ Yuan and Hu \cite{Y-H} considred equation (\ref{eq1}), with $
l=2$ or $4$ and determined the values of the integer $k$ for which equation
(\ref{eq1}) has an infinite number of positive solutions. Expanding
on the work of Yuan and Hu \cite{Y-H}, Keskin et al. in \cite{K-K-SS} and \cite{K-K-S}  considered equation (\ref{eq1}) for $l=\pm 2^{r},$ with $r$ a positive integer. They explained that in order to
determine when equation (\ref{eq1}) with $l= - 2^{r},$ has an infinite number of
positive integer solutions, one needs only to determine when the diophantine
equation
\begin{equation}\label{eq2}
x^{2}-kxy+y^{2}-2^{n}=0
\end{equation}
has an  infinite number of positive integer solutions $x$ and $y$ for certain values of the non
negative integer $n$. Similarly for $l=2^{r}$ in equation (\ref{eq1}), one needs
only to consider the diophantine equation
\begin{equation}\label{eq3}
x^{2}-kxy+y^{2}+2^{n}=0.
\end{equation}%

\noindent Keskin et al. solved equation (2) and equation (3) for $
0\leq n\leq 10$, and formulated the following conjecture.

\begin{conjecture}

(i) let $n$ be an odd integer and $n>2.$ If $k>2^{n}-2$ then equation (\ref{eq2})
has no positive integer solution. If $k \leq 2^{n} - 2$ and (\ref{eq2}) has a solution, then $k$ is even.\\
\noindent (ii) Let $n$ be an  even integer. If $k>2^{n}-2$, then  equation (\ref{eq2}) has no positive odd integer solution. If $k\leq 2^{n}-2$ and equation (\ref{eq2}) has a positive odd integer solution, then $k$ is even.
\end{conjecture}

\noindent In this paper, among other results, we will prove Conjecture 1 in Theorem 3.1, and  prove in Theorem 3.2 a result analogous to Conjecture 1.

\section{Preliminary results}

In this section, we will recall some results that we will need for the proof
of our theorems.

Let $d$ be a positive integer wich is not a perfect square, and consider the
Pell equation%
\begin{equation}\label{eq4}
x^{2}-dy^{2}=1.
\end{equation}%
It is well known that equation (\ref{eq4}) always has a positive solution when  $%
d\geq 2.$ The least positive integer solution  $x_{1}+y_{1}\sqrt{d}$ to
equation (\ref{eq4}) is called the fundamental solution, and all positive integer
solutions to (\ref{eq4}) are given by
\begin{equation*}
x_{n}+y_{n}\sqrt{d}=\left( x_{1}+y_{1}\sqrt{d}\right) ^{n},\quad \text{with}
\quad n\geq 1.
\end{equation*}
Let $C$ be a nonzero integer, and consider the Diophantine equation

\begin{equation}\label{eq5}
    u^{2}-dv^{2}=C.
\end{equation}

\noindent Suppose that $u+v\sqrt{d}$ is a solution to equation (\ref{eq5}). If $x+y\sqrt{d}$ is any solution of equation (\ref{eq4}), then
\[
\begin{array}{cc}
u^{\prime }+v^{\prime }\sqrt{d} & =\left( u+v\sqrt{d}\right) \left( x+y\sqrt{%
d}\right)  \\
& =ux+vyd+\left( yu+vx\right) \sqrt{d}%
\end{array}%
\]%
is also a solution of (\ref{eq5}).  The solution  $u^{\prime }+v^{\prime }\sqrt{d}$
is said to be associated with the solution $u+v\sqrt{d}$. The set of all
solutions associated with each other form a class of solutions of equation
(\ref{eq5}). Every class contains an infinity of solutions. We have the following lemmas.

\begin{lemma}
If $u+v\sqrt{d}$ is the fundamental solution of a class $K$ of the equation%
\begin{equation*}
u^{2}-dv^{2}=N,
\end{equation*}%
where $N$ is a positive integer and if $x_{1}+y_{1}\sqrt{d}$ is the
fundamental solution of equation (\ref{eq4}), then we have the inequalities%
\[
0 \leq v\leq \dfrac{y_{1}}{\sqrt{2\left( x_{1}+1\right) }}\sqrt{N},
\]%
and%
\[
0<\left\vert u\right\vert \leq \sqrt{\frac{1}{2}\left( x_{1}+1\right) N}.
\]
\end{lemma}

\begin{lemma}
If $u+v\sqrt{d}$ is the fundamental solution of a class $K$ of the equation $%
u^{2}-dv^{2}=-N,$ where $N$ is a positive integer and if $x_{1}+y_{1}\sqrt{d}
$ is the fundamental solution of equation (\ref{eq4}), we have the inequalities%
\[
0 < v \leq \dfrac{y_{1}}{\sqrt{2\left( x_{1}-1\right) }}\sqrt{N},
\]%
and%
\[
0 \leq \left\vert u\right\vert \leq \sqrt{\frac{1}{2}\left( x_{1}-1\right) N}.
\]
\end{lemma}

For the proof of Lemma 2.1 and Lemma 2.2, see \cite{Na}. 

\section{New results}

In this section, we will prove Conjecture 1 in Theorem 3.1, and in Theorem 3.2 a
result that is analogous to Conjecture 1 for the Diophantine equation
$$
x^{2}-kxy+y^{2}= - 2^{n}.
$$
\noindent If $k = 0$, then equation (\ref{eq2}) has finitely many solutions and equation (\ref{eq3}) has no solution. We suppose in the sequel that $k \neq 0$.
\begin{theorem}
Conjecture 1 is true
\end{theorem}

\begin{proof}

%\begin{description}
(i) Let $n > 2$ be an odd integer. If $(x,y)$ is a positive
solution of equation (\ref{eq2}), then clearly $x$ and $y$ have the same parity. If $x$ and $y$ are odd, then $k$ is even. Suppose now that $2 \mid x$, then $2 \mid y$ and vice versa, and let $x=2X$ and $y=2Y.$ Equation
(\ref{eq2}) yields 
\begin{equation}\label{eq6}
X^{2}-kXY+Y^{2}=2^{n-2}.
\end{equation}%
Again, if $X$ is even in (\ref{eq6}), then $Y$ is even and since $n - 2$ is odd, we repeat the same process untill all powers of 2 in $x$ and $y$ have been canceled. Therefore, we endup having the following equation $$X^{2}-kXY+Y^{2}=2^r ,$$ where $X$, $Y$ and $r$ are positive odd integers. Hence $k$ is clearly even. After a change of variables, equation (\ref{eq2}) with $n$ odd yields
\begin{equation}\label{eq7}
u^{2}-dv^{2}= 2^n;
\end{equation}

\noindent where $u=\left\vert x-\frac{k}{2}y\right\vert ,y=v$ and $d=\frac{k^{2}}{4}-1$. \\

\noindent Since $k$ is even, then $u$ and $v$ are positive integers. If $k = 2$, equation (\ref{eq7}) implies that $u^2 = 2^n$, which is impossible. Hence, $k > 2$, whereupon $d > 1$. The solution $\frac{k}{2}+\sqrt{\frac{k^{2}}{4}-1}$ is the fundamental solution to the Diophantine equation
\[
x^{2}-dy^{2}=1,\quad \text{where\quad }d=\frac{k^{2}}{4}-1.
\] 
If equation (\ref{eq2}) has a positive solution with $n$ an odd positive integer,
then equation (\ref{eq7}) has a positive solution. If $u+v\sqrt{d}$ is the
fundamental solution of a class $K$ of equation (\ref{eq7}), then Lemma 2.1 implies
that
\[
0\leq v\leq \frac{1}{\sqrt{2\left( \frac{k}{2}+1\right) }}\sqrt{2^{n}}.
\] If $v = 0$, then equation (\ref{eq7}) yields $u^{2} = 2^{n}$, which is impossible. Therefore, $v \geq 1$ and the inequality above implies that  $\sqrt{k+2}\leq \sqrt{2^{n}},$ i.e. $k\leq 2^{n}-2.$\\ 

\noindent (ii) Let $n$ be a positive even integer, and suppose that $(x,y)$ is a solution to
$x^{2}-kxy+y^{2}=2^{n}$. If $x$ and $y$ are odd, then clearly $k$ is even. Hence equation (\ref{eq2}) yields $u^{2}-dv^{2}=2^{n},$ where $
u=\left\vert x-\frac{k}{2}y\right\vert ,v=y$ and $d=\frac{k^{2}}{4}-1$. Since $
k\neq 0$, then $k\geq 2,$ and $d$ is a non negative integer. Lemma 2.1 implies that
\[
0\leq v\leq \frac{1}{\sqrt{2\left( \frac{k}{2}+1\right) }}\sqrt{2^{n}}.
\]
The solution $\frac{k}{2}+\sqrt{\frac{k^{2}}{4}-1}$ is the fundamental solution
to $x^{2}-dy^{2}=1,$ where $d=\frac{k^{2}}{4}-1.$ If $v = 0$, then equation (\ref{eq7}) yields $u = 2^{n/2}$ and all solutions in the same classe as $(2^{n/2} , 0)$ are even. Hence, we suppose $v \geq 1$, and the last inequality implies that  $\sqrt{2\left(\frac{k}{2}+1\right) }\leq \sqrt{2^{n}},$ i.e. $k\leq 2^{n}-2$.  Therefore, if $n$ is
even and $k>2^{n}-2$, equation (\ref{eq2}) has no positive odd solution and if  $k\leq 2^{n}-2$ and equation (\ref{eq2}) has a  positive odd
integer solution, then $k$ is even.
%\end{description}
\end{proof}

\begin{theorem}

%\begin{description}
(i) Let $n$ be an odd integer and $n>2.$ If $k>2^{n}+2,$ then the equation $x^{2}-kxy+y^{2}=2^{n}$ has no positive integer solution. If $k\leq 2^{n}+2,$ and the equation $x^{2}-kxy+y^{2}= - 2^{n}$ has a solution, then $k$ is even.\\ 

\noindent (ii) Let $n$ be an nonzero even integer. If $k>2^{n}+2,$ then the equation $%
x^{2}-kxy+y^{2}= - 2^{n}$ has no positive odd integer solution. If $k\leq
2^{n}+2$ and the equation $x^{2}-kxy+y^{2}= - 2^{n}$ has a positive odd integer
solution, then $k$ is even and $2$ divide exactly $k$.
%%\end{description}
\end{theorem}

\begin{proof}

%\begin{description}
(i) let $n$ be a positive odd integer and $n>2$. Using the same
reasoning as in the proof of Theorem 3.1, without loss of
generality, we can suppose that the solutions $x$ and $y$ to (\ref{eq1}) are odd. Hence $k$ is even. Again, the same method in the proof of Theorem 3.1 and Lemma 2.2
imply that
\[
1\leq v\leq \frac{1}{\sqrt{2\left( \frac{k}{2}-1\right) }}\sqrt{2^{n}}
\]%
whereupon, $\sqrt{k-2}\leq \sqrt{2^{n}}$ i.e. $k\leq 2^{n}+2$.\\

\noindent (ii) Suppose that $n$ is even and that the equation $%
x^{2}-kxy+y^{2}= - 2^{n}$ has a positive integer solution, then clearly $k$ is
even because $n\geq 1,$ (the case $n=0$ has been settled in \cite{K-K-SS}). Again the
same method in the proof of Theorem 3.1 and Lemma 2.2 imply that $k\leq 2^{n}+2$ and $k$ even. If $(x,y)$ is an  odd solution to equation (\ref{eq1}), then taking  $
x^{2}-kxy+y^{2}=2^{n}$ modulo 4 implies that 2 divide exactly $k$. 
%\end{description}
\end{proof}

\begin{remark}
It was proved in \cite{K-K-S}, that the diophantine equation $x^{2}-kxy+y^{2}=2^{n}$
with $k=2^{n}-2$ has infinitly many solutions and in \cite{K-K-SS}, that the
Diophantine equation $x^{2}-kxy+y^{2}=-2^{n},$ with $k=2^{n}+2$ has
infinitly many solutions. Hence the bounds of $k$ in Theorem 3.1 and Theorem 3.2
are sharp.
\end{remark}
\begin{theorem}
%\begin{description}
(i) Let $n>2$ is an odd integer and $p$ a prime such that $\left(
\frac{2}{p}\right) =-1$. If equation (\ref{eq2}) has a positive solution, then $
\frac{k}{2} \not\equiv \pm 1\mod p.$ In particular $k$ is a multiple of $3.$ \\ 

\noindent (ii) Let $n>2$  an odd integer, and $p$ a prime such that $\left( \frac{%
2}{p}\right) =1$. If equation (\ref{eq3}) has a positive solution, then $\frac{k}{2}%
\not\equiv \pm 1\mod p.$ 
%\end{description}
\end{theorem}
\begin{proof}

%\begin{description}
(i) If $n>2$ is an odd integer and equation (\ref{eq2}) has a positive
solution, then the proof of Theorem 3.1 implies that $k$ is even and the
Diophantine equation $u^{2}-\left( \frac{k^{2}}{4}-1\right) v^{2}=2^{n}$ is
solvable. Hence if $p$ is an odd prime such that  $\left( \frac{2}{p}\right)
=-1,$ then $\frac{k}{2}\not\equiv \pm 1\mod p$. By taking $p =3$, we obtain that  $k$ is a multiple of $3.$\\

\noindent (ii) The proof of (ii) is similar to (i) and will be omitted.
%\end{description}
\end{proof}


\begin {thebibliography}{1000}
\bibitem {K} R. Keskin,  {\it Solutions of some quadratic  Diophantine equations}, Comput. Math. Appl. {\bf 60} (2010), 2225--2230.
\bibitem {K-K-SS} R. Keskin, O. Karaatli and Z. Siar, {\it On the Diophantine Equation $x^{2}-kxy+y^{2}+2^{n}=0$,} Miskolc Math. Notes, {\bf 13} (2012), 375-388.
\bibitem {K-K-S} R. Keskin, Z. Siar and O. Karaatli, {\it On the Diophantine Equation $x^{2}-kxy+y^{2} - 2^{n}=0$,} Czechoslovak Mathematical Journal, {\bf 63} (138)(2013), 783-797.
\bibitem {M-Z} A. Marlewski and P. Zarzycki {\it Infinitely many solutions of the Diophantine equation $x^{2}-kxy+y^{2} + x =0$,} Comput. Math. Appl. {\bf 47} (2004), 115--121.
\bibitem{Na} T. Nagell, {\it  Introduction to number theory}, John Wiley \& sons,  Inc.,  New York, Stockholm, 1951.
\bibitem {Y-H} P. Yuan and Y. Hu  {\it On the Diophantine equation $x^{2}-kxy+y^{2}+lx=0$, $l \in \{1,\, 2\, 4\}$}, Comput. Math. Appl. {\bf 61} (2011), 573--577.


\end{thebibliography}
\end{document}